\documentclass[a4paper,11pt]{amsart}
\usepackage{amsmath,amssymb}
\newtheorem{theorem}{Theorem}
\newtheorem{lemma}[theorem]{Lemma}
\newtheorem{proposition}[theorem]{Proposition}
\newtheorem{corollary}[theorem]{Corollary}
\theoremstyle{definition}

\theoremstyle{remark}
\newtheorem{remark}{Remark}
\DeclareMathOperator{\sech}{sech}
\newcommand{\R}{\mathbb{R}}

\newcommand{\N}{\mathbb{N}}
\newcommand{\Z}{\mathbb{Z}}
\newcommand{\la}{\langle}
\newcommand{\ra}{\rangle}
\newcommand{\pd}{\partial}

\begin{document}
\title{Asymptotic stability of lattice solitons in the energy space}
\author{Tetsu Mizumachi}
\address{Faculty of Mathematics, Kyushu University,
Hakozaki 6-10-1, 812-8581 Japan}
\email{mizumati@math.kyushu-u.ac.jp}
\begin{abstract}
Orbital and asymptotic stability for 1-soliton solutions to the Toda lattice
equations as well as small solitary waves to the FPU lattice equations are
established in the energy space.
Unlike analogous Hamiltonian PDEs, the lattice equations do not
conserve momentum.
Furthermore, the Toda lattice equation is a bidirectional model that does not
fit in with existing theory for Hamiltonian system by Grillakis, Shatah and
Strauss. 

To prove stability of 1-soliton solutions, we split a  solution around
a 1-soliton into a small solution that moves more slowly than the main
solitary wave, and an exponentially localized part. We apply a decay estimate
for solutions to a linearized Toda equation which has been recently proved by
Mizumachi and Pego to estimate the localized part.
We improve the asymptotic stability results for FPU lattices in a weighted
space obtained by Friesecke and Pego.
\end{abstract}
\maketitle
\section{Introduction}
\label{sec:intro}
In this paper, we study asymptotic stability of solitary waves
to a class of Hamiltonian systems of particles connected by nonlinear springs. A typical model of these lattice
is Toda lattice
\begin{equation}
\label{eq:Toda}
\ddot{q}(t,n)=e^{-(q(t,n)-q(t,n-1))}-e^{-(q(t,n+1)-q(t,n))}
\quad\text{for $t\in\R$ and $n\in\Z$},
\end{equation}
where $q(t,n)$ denotes the displacement of the $n$-th particle at time $t$
and $\;\dot{}\;$ denotes differentiation with respect to $t$.
Let $p(t,n)=\dot{q}(t,n)$,  $r(t,n)=q(t,n+1)-q(t,n)$,
$u(t,n)={}^t\!(r(t,n),p(t,n))$ and $V(r)=e^{-r}-1+r$. 
Toda lattice \eqref{eq:Toda} is an integrable system with the Hamiltonian
$$H(u(t))=\sum_{n\in\Z}\left(\frac12p(t,n)^2+V(r(t,n))\right),$$
(see \cite{Fl1}) and it can be rewritten as
\begin{equation}
  \label{eq:Toda2}
  \frac{du}{dt}=JH'(u),
\end{equation}
where $$J=\begin{pmatrix} 0 & e^\partial-1 \\ 1-e^{-\partial} & 0
\end{pmatrix},$$
and $e^{\pm\pd}=e^{\pm\frac{\pd}{\pd n}}$ are the shift operator defined
by $(e^{\pm\pd})f(n)=f(n\pm1)$ for every sequence $\{f(n)\}_{n\in\Z}$
and $H'$ is the Fr\'{e}chet derivative of $H$ in $l^2\times l^2$.
\par

Toda lattice \eqref{eq:Toda2} has a two-parameter family of solitary waves
$$\mathcal{M}=\left\{u_{c}(t+\delta) \,\bigm|\, c>1,\;\delta\in\R\right\},$$  
where $u_c(t,n)=\tilde{u}_c(n-ct)$, $\tilde{u}_c(x)=
(\tilde{r}_c(x),\tilde{p}_c(x))$ and
\begin{align}
\label{eq:q}
& \tilde{q}_c(x)=\log\frac{\cosh\{\kappa(x-1)\}}{\cosh\kappa x},\\
\label{eq:u}
& \tilde{p}_c(x)=-c\pd_x\tilde{q}_c(x),\quad 
\tilde{r}_c(x)=\tilde{q}_c(x+1)-\tilde{q}_c(x),
\end{align}
and $\kappa=\kappa(c)$ is a unique positive solution of $c=\sinh\kappa/\kappa$.
\par

Friesecke and Pego \cite{FP2,FP3} have proved asymptotic stability of solitary
waves to FPU lattice in a weighted space assuming an exponential linear
stability property (H1) below.

To state the assumption explicitly, we introduce several notations.
Let $l^2_a$ be a Hilbert space of $\R^2$-sequences equipped with the norm
$$\|u\|_{l^2_a}=\left(\sum_{n\in\N}e^{2an}|u(n)|^2\right)^{1/2}.$$ 
Let $\la u,v\ra:=\sum_{n\in\Z}(u_1(n)u_2(n)+v_1(n)v_2(n))$
for $\R^2$-sequences $u=(u_1,u_2)$ and $v=(v_1,v_2)$ and
$\|u\|_{l^2}=(\la u,u\ra)^{1/2}$.

\begin{itemize}
\item [(H1)]
Let $a>0$ be a small number. There exist positive numbers $K$ and $\beta$
such that if
\begin{equation}
  \label{eq:secular}
\la v(s),J^{-1}\dot{u}_c(s)\ra
=\la v(s),J^{-1}\pd_cu_c(s)\ra=0,  
\end{equation}
then a solution to 
\begin{align}
\label{eq:lToda}
& \frac{dv}{dt}=JH''(u_c(t))v
\end{align}
satisfies 
\begin{equation} \label{eq:decay}
\| e^{a(\cdot-ct)}u(t,\cdot)\|_{l^2} \le K e^{-\beta(t-s)}
\| e^{a(\cdot-cs)}u(s)\|_{l^2}
\quad\text{for every $t\ge s$.}
\end{equation}
\end{itemize}
\begin{remark}
Solutions $\dot{u}_c(t)$ and $\pd_cu_c(t)$ to \eqref{eq:lToda} correspond to
infinitesimal changes on $t$ and $c$ and they do not decay as $t\to\infty$.
Since $J^{-1}\dot{u}_c(t)$ and $J^{-1}\pd_cu_c(t)$  are the corresponding
neutral modes to the adjoint equation
$$\frac{dw}{dt}=H''(u_c(t))Jw,$$
the condition (H1) says that a solution to
\eqref{eq:lToda} decays exponentially as $t\to\infty$
if it does not include neutral modes $\dot{u}_c(t)$ and $\pd_cu_c(t)$.
\end{remark}
\begin{remark}
In \eqref{eq:secular}, we set
  \begin{equation*}
J^{-1}= \begin{pmatrix}
  0 & \sum_{k=-\infty}^0 e^{k\pd} \\ \sum_{k=-\infty}^{-1}e^{k\pd} & 0
\end{pmatrix}    
  \end{equation*}
so that $J^{-1}$ is a bounded operator in $l^2_{-a}$.
(Note that $u$ decays exponentially as $n\to\pm\infty$ if $u\in l^2_{\pm a}$
and $a>0$ and that $\|e^{-\pd}u\|_{l^2_{-a}}=e^{-a}\|u\|_{l^2_{-a}}$.)
Since $\dot{u}_c$ and $\pd_cu_c$ decay like $e^{-2\kappa|n-ct|}$ as
$n\to\pm\infty$, we have $J^{-1}\dot{u}_c$, $J^{-1}\pd_cu_c\in l^2_{-a}$ for
every $a\in (0,2\kappa(c))$.
\end{remark}
Friesecke and Pego prove in \cite{FP2} that solitary waves to FPU lattice
are asymptotically stable in $l^2_a$ if (H1) holds.
They have also proved in \cite{FP3,FP4} that small solitary waves of
FPU lattice can be approximated by KdV solitons and
that they satisfy (H1).
In \cite{MP}, we use the linearized B\"{a}cklund transformation to show that
every 1-soliton of Toda lattice satisfies (H1) and prove that it is
asymptotically stable in $l^2_a$ without assuming smallness of solitons.

Our goal in the present paper is to prove asymptotic stability of 1-solitons
in $l^2$.
\begin{theorem}
  \label{thm:1}
Let $c_0>1$, $\tau_0\in\R$ and let $u(t)$ be a solution to \eqref{eq:Toda2}
with $u(0)=u_{c_0}(\tau_0)+v_0$.
For every $\varepsilon>0$, there exists a positive number $\delta>0$ satisfying
the following: If $\|v_0\|_{l^2}<\delta$, there exist constants $c_+>1$ and
$\sigma\in(1,c_+)$ and a $C^1$-function $x(t)$ such that
\begin{align*}
& \|u(t)-\tilde{u}_{c_0}(\cdot-x(t))\|_{l^2}< \varepsilon,\\
& 
\lim_{t\to\infty}\left\|
u(t)-\tilde{u}_{c_+}(\cdot-x(t))\right\|_{l^2(n\ge \sigma t)}=0,
\\ 
& \sup_{t\in\R}\left(|c(t)-c_0|+|\dot{x}(t)-c_0|\right)=O(\|v_0\|_{l^2}),
\\ 
& \lim_{t\to\infty}c(t)=c_+,\quad \lim_{t\to\infty}\dot{x}(t)=c_+.
\end{align*}
\end{theorem}
\begin{remark}
By a simple computation, we see $dH(u_c)/dc>0$ and $\lim_{c\to1}H(u_c)=0$
(see e.g. \cite{To}).
So we have arbitrary small 1-solitons in $l^2$. However, small solitary waves
do not belong to an exponentially weighted space if $c$ is close to
$1$ because $u_c(t)$ decays like $e^{-2\kappa(c)|n-x(t)|}$ as $n\to\infty$
and $\lim_{c\downarrow1}\kappa(c)=0$.
Thus from Friesecke and Pego \cite{FP1,FP2,FP3,FP4} and Mizumachi and Pego
\cite{MP}, we cannot see whether a solitary wave can be stable under
perturbations which include small solitary waves. Theorem \ref{thm:1} and 
Theorem \ref{thm:2} below insist that a solitary wave does not collapse by
small perturbations including other solitary waves.
\end{remark}
Since Benjamin \cite{Be} and Bona \cite{Bo} studied stability of KdV
$1$-solitons, a lot of results have been obtained on stability of solitary
waves to infinite dimensional Hamiltonian systems (see \cite{Cazenave} and
references therein). In those results, they utilized the fact that the
Hamiltonian systems have another conservation law (like momentum for KdV and
charge for NLS) and a solitary wave solution is a local minimizer of the
Hamiltonian among solutions whose momentum or charge is the same as the
solitary wave solution.
\par
However, Toda and FPU lattices are bidirectional models like Boussinesq
equations (see \cite{BoCS1,BoCS2,PSW}) such that a solitary wave solution is
a saddle point of the sum of Hamiltonian and the momentum multiplied by
the speed of the
solitary wave whose second variation has infinite dimensional indefiniteness.
Furthermore, a solution to Toda lattice does not conserve momentum in general
because Noether's theorem is not applicable to spatial variable $n\in\Z$.
Hence stability of solitary waves does not follow from the theory of
Hamiltonian system by Grillakis, Shatah and Strauss \cite{GSS1,GSS2} and
Shatah and Strauss \cite{ShSt}.
For the same reason, it is not possible to use a Liouville theorem
like \cite{MM1} to prove asymptotic stability of solitary waves.
\par

Luckily, solitary waves for a class of lattice equations including the Toda
lattice equation separate from each other as $t\to\infty$.
As can be seen from \eqref{eq:q} and \eqref{eq:u}, 
speed of solitary waves which move to the right is larger than $1$ and
the larger a solitary wave is the faster it moves, whereas the absolute value of
group velocities are less than $1$. So a solution to \eqref{eq:Toda2} is
decoupled into a train of solitary waves and a remainder term as $t\to\infty$.
\par
Friesecke and Pego \cite{FP1,FP2,FP3,FP4} utilized this fact and prove
asymptotic stability of solitary waves to FPU lattice in an exponentially
weighted space. 
They decompose a solitary wave as
\begin{equation}
  \label{eq:decomp}
  \begin{split}
& u(t)=u_{c(t)}(\gamma(t))+v(t) =\tilde{u}_{c(t)}(\cdot-x(t))+v(t),\\
&  x(t)=c(t)\gamma(t),
  \end{split}
\end{equation}
where $u_{c(t)}(\gamma(t))$ denotes a main solitary wave, and
$c(t)$ and $x(t)$ are modulation parameters of the speed
and the phase shift of the main wave, respectively. They prove that a solution
which lies in a neighborhood of $\mathcal{M}$ is absorbed into $\mathcal{M}$ 
exponentially in $l^2_a$-norm as $t\to\infty$. Their proof basically follows
the idea of Pego and Weinstein \cite{PW} and impose the symplectical
orthogonality condition \eqref{eq:secular} on $v$. 
One of the difficulty to use their method in the energy space is that
$J^{-1}\pd_cu_c$ tends to a nonzero constant as $n\to\infty$ and
\eqref{eq:secular} is not well defined for $v\in l^2$. 
\par
Our strategy is to decompose $v(t)$ into the sum of a small solution $v_1(t)$
of \eqref{eq:Toda2} and $v_2(t)$ that is driven by an interaction of $u_c$ and
dispersive part of the solution.
Since $v_2(t)$ is exponentially localized in front, we can estimate $v_2(t)$ by
using exponential linear stability \eqref{eq:decay}. Since $v_1(t)$ moves more
slowly than the main solitary waves, it locally tends to $0$ around
the solitary wave. To fix the decomposition, we impose the constraint
$$\la v, J^{-1}\dot{u}_{c}(\gamma)\ra=
\la v_2,J^{-1}\pd_cu_c(\gamma)\ra=0$$ instead of \eqref{eq:secular}.
\par
Recently, Martel and Merle \cite{MM2} give a direct proof of the asymptotic
stability results in $H^1(\R)$ for generalized KdV solitons based on a
virial identity (which first appeared in Kato \cite{K}).
Because the Toda lattice and KdV equations have a similarity that
the dominant solitary wave outruns and is separated from other part of
solutions as $t\to\infty$, their idea seems promising.
We prove a \textit{virial lemma} [Lemma \ref{lem:monotonicity} in
Section \ref{sec:3}] for $v_1(t)$ and  apply local energy decay estimates
for other part of the solution instead of proving a virial lemma around
solitary waves.This enables us to prove our results without numerics whereas
\cite{MM1,MM2} need some numerical computation to prove positivity of a
quadratic form. We expect our proof is applicable also for Hamiltonian PDEs
like KdV equation by using the renormalization method by Ei \cite{Ei} and
Promislow \cite{Prom} (see \cite{Mi} for an application to the generalized KdV
equation in a weighted space).
\par
Now, let us consider asymptotic stability of solitary waves to FPU lattice
equations. It is interesting to see whether solitary waves to non-integrable
lattices are robust to perturbations in the energy class.
Let $u(t,n)={}^t\!(r(t,n),p(t,n))$ be a solution to 
\begin{equation}
  \label{eq:FPU}
\frac{du}{dt}=JH_F'(u)\quad\text{for $t\in\R$,}
\end{equation}
where
$$H_F(u(t))=\sum_{n\in\Z}\left(\frac12p(t,n)^2+V_F(r(t,n))\right),$$
and $V_F$ is a potential satisfying
\begin{equation}
  \label{eq:H2}
V_F\in C^4(\R;\R),\quad V_F(0)=V_F'(0)=0, \quad V_F''(0)>0, \quad
V_F'''(0)\ne0. \tag{H2}
\end{equation}
If $c>c_s:=\sqrt{V_F''(0)}$ and $c$ is sufficiently close to $c_s$,
Friesecke and Pego \cite{FP1} show that there exists a unique solution
$\tilde{u}_c(x)$ 
\begin{equation}
  \label{eq:solitary}
-c\pd_x\tilde{u}_c(x)=JH_F'(\tilde{u}_c(x))
\quad\text{for $x\in\R$}  
\end{equation}
up to translation and its profile is close to that of a KdV
soliton. We remark that a solitary wave solution $\tilde{u}_c(n-ct)$ has small
amplitude and satisfies $dH(\tilde{u}_c)/dc>0$ if $c>c_s$ and $c$ is close to
$c_s$. See Friesecke and Wattis \cite{FW} for existence of large solitary
waves. Friesecke and Pego have proved in \cite{FP4} that small solitary wave
solutions of \eqref{eq:FPU} satisfy (H1) and are asymptotically stable in
$l^2_a$. Assuming (H2), we can prove orbital and asymptotic stability
of small solitary waves in $l^2$ exactly in the same way as Toda lattice.
\begin{theorem}
  \label{thm:2}
Suppose (H2). Let $\delta_*$ be a small positive number and let
$c_0\in(c_s,c_s+\delta_*)$ and $\tau_0\in\R$. Let $u(t)$ be a solution to
\eqref{eq:FPU} with $u(0)=u_{c_0}(\tau_0)+v_0$.
Then for every $\varepsilon>0$, there exists a $\delta>0$ satisfying the
following: 
If $\|v_0\|_{l^2}<\delta$, there exist constants $c_+>c_s$ and
$\sigma\in(c_s,c_+)$ and a $C^1$-function $x(t)$ such that
\begin{align*}
& \|u(t)-\tilde{u}_{c_0}(\cdot-x(t))\|_{l^2}< \varepsilon,\\
& \lim_{t\to\infty}\left\|u(t)-\tilde{u}_{c_+}(\cdot-x(t))
\right\|_{l^2(n\ge \sigma t)}=0,\\
& \sup_{t\in\R}\left(|c(t)-c_0|+|\dot{x}(t)-c_0|\right)=O(\|v_0\|_{l^2}),\\
& \lim_{t\to\infty}c(t)=c_+,\quad \lim_{t\to\infty}\dot{x}(t)=c_+.
\end{align*}
\end{theorem}
\par

Our plan of the present paper is as follows. In Section \ref{sec:2}, we introduce a
variant of the secular term condition for solutions in the energy class and
some estimates that will be used later.
In Section \ref{sec:3}, we derive modulation equations of $x(t)$ and $c(t)$ and prove
\begin{equation}
  \label{eq:i2}
\dot{c}(t)=O(\|v_1(t)\|_{W}^2+\|v_2(t)\|_{X}^2)
\end{equation}
for some weighted space $W\subset l^2_a\cap l^2_{-a}$ and $X\subset l^2_a$.
On the other hand, we show that
\begin{equation}
  \label{eq:i3}
\int_0^\infty(\|v_1(t)\|_{W}+\|v_2(t)\|_X)^2dt \lesssim \|v_0\|^2_{l^2}  
\end{equation}
by using a virial lemma for $v_1(t)$ and a local energy decay estimate
(Corollary \ref{cor:1} in Section \ref{sec:2}) for $v_2(t)$.
Combining \eqref{eq:i2} and \eqref{eq:i3} with
\begin{equation}
  \label{eq:i1}
\|v(t)\|_{l^2}^2\le C(\|v_0\|_{l^2}+|c(t)-c_0|),
\end{equation}
which follows from the convexity of the Hamiltonian and the orthogonality
condition, we will prove Theorem \ref{thm:1}.
In Section \ref{sec:4}, we give a brief proof of Theorem \ref{thm:2}.
\par
Finally, let us introduce some notations.
For a Banach space $X$, we denote by $B(X)$ the space of 
all linear continuous operators from $X$ to $X$.
We use $a\lesssim b$ and $a=O(b)$ to mean that there exists a positive constant
such that $a\le Cb$.

\section{Preliminaries}
\label{sec:2}
Let $u(t)$ be a solution to \eqref{eq:Toda2} which lies in a tubular
neighborhood of $\mathcal{M}$. We decompose $u(t)$ as \eqref{eq:decomp}.
Since $\dot{u}_c=-c\pd_x\tilde{u}_c(\cdot-ct)=JH'(u_c)$,
it follows from
\eqref{eq:q} and \eqref{eq:u} that 
\begin{align*}
\frac{d}{dt}u_{c(t)}(\gamma(t))=& \dot{c}(t)\pd_c\tilde{u}_{c(t)}(n-x(t))
-\dot{x}(t)\pd_x\tilde{u}_{c(t)}(n-x(t))
\\=&  JH'(u_c(t))+ \dot{c}(t)\pd_cu_c(\gamma(t))
+\frac{\dot{x}(t)-c(t)}{c(t)}\dot{u}_{c(t)}(\gamma(t)).
\end{align*}
Thus by the definition of $v$,
\begin{equation}
  \label{eq:v}
\frac{dv}{dt}=JH''(u_{c(t)}(\gamma(t)))v(t)+l_1(t)+N_1(t),
\end{equation}
where
\begin{align*}
l_1(t)=& -\dot{c}(t)\pd_cu_{c(t)}(\gamma(t))
-\frac{\dot{x}(t)-c(t)}{c(t)}\dot{u}_{c(t)}(\gamma(t)),
\\
N_1(t)=& J\left\{H'(u_{c(t)}(\gamma(t))+v(t))-H'(u_{c(t)}(\gamma(t)))
-H''(u_{c(t)}(\gamma(t)))v(t)\right\}.
\end{align*}
Let $P_c(t)$ be a spectral projection associated with a subspace
of neutral modes $\operatorname{span}\{\dot{u}_c(t), \pd_cu_c(t)\}$
and let $Q_{c}(t)=1-P_c(t)$.
Then for $v\in l^2_a$ $(0<a<2\kappa(c))$,
$$P_c(t)v= \theta(c)\la v,J^{-1}\dot{u}_c(t)\ra \pd_cu_c(t)-\theta(c)\la v,
J^{-1}\pd_cu_c(t)\ra\dot{u}_c(t),$$
where $\theta(c)=\left(dH(u_c)/dc)\right)^{-1}$.
We remark that the projections $P_c(t)$ and $Q_c(t)$ cannot be defined on
$l^2$ because $J^{-1}\pd_cu_c$ does not decay as $n\to\infty$.
\par
Now, we decompose $v(t)$ into the sum of a small solution to
\eqref{eq:Toda2} and a remainder term that belongs to $l^2_a$ for some
$a>0$. More precisely, we put $v(t)=v_1(t)+v_2(t)$, where
\begin{equation}
  \label{eq:v1}
\left\{
  \begin{aligned}
&  \frac{dv_1}{dt}=JH'(v_1),\\
&   v_1(0)=v_0,
  \end{aligned}\right.
\end{equation}
and $v_2(t)$ is a solution to
\begin{equation}
  \label{eq:v2}
\left\{
  \begin{aligned}
& \frac{dv_2}{dt}= JH''(u_{c(t)}(\gamma(t)))v_2+l_1(t)+N_2(t),
\\ & v_2(0)=\varphi_{c_0}(\tau_0)-\varphi_{c(0)}(\gamma(0)),
  \end{aligned}\right.
\end{equation}
where $N_2(t)=N_1(t)-JH'(v_1(t))+JH''(u_{c(t)}(\gamma(t)))v_1$.
To fix the decomposition, we will impose the constraint
\begin{align}
\label{eq:orth1}
& \la v(t), J^{-1}\dot{u}_{c(t)}(\gamma(t))\ra=0,\\
\label{eq:orth2}
& \la v_2(t),J^{-1}\pd_cu_{c(t)}(\gamma(t))\ra=0.
\end{align}
We remark that $u(t)-v_1(t)$ remains in $l^2_a$ for every $0\le a<2\kappa(c_0)$
and $t\in\R$. More precisely, we have the following.
\begin{proposition}
  \label{prop:1}
Let $c_0>1$, $\tau_0\in\R$ and $v_0\in l^2$.
Let $u(t)$ be a solution to \eqref{eq:Toda2} satisfying
$u(0)=u_{c_0}(\tau_0)+v_0$ and let $v_1(t)$ be a solution to \eqref{eq:v1}.
Then
$$u(t)\in C^2(\R;l^2)\quad\text{and}\quad u(t)-v_1(t)\in C^2(\R;l^2_a)
\quad\text{for $0\le a<2\kappa(c_0)$.}$$
\end{proposition}
\begin{proof}
By \cite{FP2}, we have $u$, $v_1\in C^2(\R;l^2)$.
Let $v_3(t)=u(t)-v_1(t)$. Then $v_3(0)\in \cap_{0\le a<2\kappa(c_0)} l^2_a$
and 
\begin{equation}
  \label{eq:v31}
\frac{dv_3}{dt}=J(H'(u)-H'(v_1)).  
\end{equation}
Let $u(t)={}^t\!(r(t),p(t))$, $v_1(t)={}^t\!(r_1(t),p_1(t))$ and let
$$F(u,v_1)=
\begin{pmatrix}
  \frac{V'(r)-V'(r_1)}{r-r_1} & 0 \\ 0 & 1 \end{pmatrix}.$$
Then we have $F(u,v_1)\in C^1(\R;B(l^2_a))$ for every $a\in[0,2\kappa(c_0))$
and \eqref{eq:v31} can be rewritten as 
\begin{equation}
  \label{eq:v32}
\frac{dv_3}{dt}=JF(u,v_1)v_3.
\end{equation}
By \cite[Appendix A]{FP2}, we see that there exists a unique solution 
$v_3\in C^2(\R;l^2\cap l^2_a)$ to \eqref{eq:v32}
for every $a\in[0,2\kappa(c_0))$. Thus we prove $u-v_1\in C^2(\R;l^2_a)$
for every $a\in[0,2\kappa(c_0))$.
\end{proof}
If $u(t)$ and $u(t)-v_1(t)$ lie in a tubular neighborhood of a solitary
wave in $l^2$ and $l^2_a$ respectively, we can find modulation parameters
$c(t)$ and $\gamma(t)$ satisfying \eqref{eq:orth1} and \eqref{eq:orth2}.
\begin{lemma}
 \label{lem:decomp}
Let $c_0>1$, $\tau_0\in\R$, $\gamma_0(t)=t+\tau_0$ and $a\in(0,2\kappa(c_0))$.
Let $u(t)$ be a solution to \eqref{eq:Toda2} and let $v_1(t)$ be a solution to
\eqref{eq:v1}. Then there exist positive numbers $\delta_0$ and $\delta_1$
satisfying the following: If
$$\sup_{t\in[T_1,T_2]}\left(\|u(t)-u_{c_0}(\gamma_0(t))\|_{l^2}
+e^{-ac_0\gamma_0(t)}\|u(t)-u_{c_0}(\gamma_0(t))-v_1(t)\|_{l^2_a}\right)
<\delta_0$$
for some $0\le T_1\le T_2\le\infty$,
there exists $(c(t),\gamma(t))\in C^2([T_1,T_2];\R^2)$ satisfying
\eqref{eq:decomp}, \eqref{eq:orth1}, \eqref{eq:orth2} and
$$\sup_{t\in[T_1,T_2]}\left(|\gamma(t)-\gamma_0(t)|+|c(t)-c_0|\right)
<\delta_1.$$
Especially, it holds
$|c(0)-c_0|+|\gamma(0)-\tau_0|=O(\|v_0\|_{l^2}).$
\end{lemma}
\begin{proof}
  Put
\begin{gather}
F_1(u,\tilde{u},c,\gamma)
:=\la u-u_c(\gamma),J^{-1}\dot{u}_c(\gamma)\ra,\\
F_2(u,\tilde{u},c,\gamma)
:=\la \tilde{u}-u_c(\gamma),J^{-1}\pd_cu_c(\gamma))\ra.
\end{gather}
Then
\begin{align*}
\frac{\pd(F_1,F_2)}{\pd(c,\gamma)}(u_{c_0}(\gamma_0),u_{c_0}(\gamma_0),
c_0,\gamma_0)
= -\left(\frac{d}{dc}H(u_{c_0})\right)^2\ne0.  
\end{align*}
Let 
$U(\delta_0)=$$\{(u,\tilde{u})\in l^2\times l^2_a: 
\|u-u_c(\gamma_0)\|_{l^2}+
e^{-ac\gamma_0}\|\tilde{u}-u_c(\gamma_0)\|_{l^2_a}<\delta_0\}$ and
$B(\delta_1):=\{(c,\gamma)\in\R^2:|c-c_0|+|\gamma-\gamma_0|<\delta_1\}$.
Using the implicit function theorem, we see that there exists positive numbers
$\delta_0$ and $\delta_1$ and a mapping
$$\Phi:U(\delta_0)\ni (u,\tilde{u})\mapsto (c,\gamma)\in B(\delta_1)$$
satisfying 
$F_1(u,\tilde{u},\Phi(u,\tilde{u}))=F_2(u,\tilde{u},\Phi(u,\tilde{u}))=0$.
Since $F_1$ and $F_2$ are $C^2$ in $(u,\tilde{u},\gamma,c)\in
U(\delta_0)\times B(\delta_1)$, we have
$\Phi\in C^2(U(\delta_0))$.
\par
Let $(c(t),\gamma(t))=\Phi(u(t),u(t)-v_1(t))$ for $t\in[T_1,T_2]$.
Then $c(t)$ and $\gamma(t)$ satisfy \eqref{eq:orth1} and \eqref{eq:orth2}
and are of class $C^2$ because $\Phi\in C^2(U(\delta_0))$ and 
$(u(t),u(t)-v_1(t))\in C^2(\R;U(\delta_0))$.
Furthermore, we have 
\begin{align*}
& |c(t)-c_0|+|\gamma(t)-\gamma_0(t)|
\\ \lesssim &
\|u(t)-u_{c_0}(\gamma_0(t))\|_{l^2}+
e^{-ac_0\gamma(t)}\|u(t)-u_{c_0}(\gamma_0(t))-v_1(t)\|_{l^2_a}.
\end{align*}
Especially for $t=0$, we have
$|c(0)-c_0|+|\gamma(0)-\tau_0|=O(\|v_0\|_{l^2}).$
This completes the proof of Lemma \ref{lem:decomp}.
\end{proof}

To estimate the exponentially decaying part of a solution, we will use the
following decay estimate for non-autonomous linearized equations.
\begin{lemma}[\cite{FP3,MP}]
  \label{lem:linear-est}
Let $c_0>1$, $a\in(0,2\kappa(c_0))$ and $b(a):=ca-2\sinh(a/2)$.
Let $U_0(t,\tau)\varphi$ be a solution to
\begin{equation}
  \label{eq:LToda}
\left\{
  \begin{aligned}
& \frac{dv}{dt}=JH''(u_{c_0})v.\\
& v(\tau)=\varphi.
  \end{aligned}
\right.
\end{equation}
Then for every $b\in(0,b(a))$, there exists a positive number $K$ such that
for every $\varphi\in l^2_a$ and $t\ge \tau$,
$$e^{-ac_0(t-\tau)}\|U_0(t,\tau)Q_c(\tau)\varphi\|_{l^2_a}
\le Ke^{-b(t-\tau)}\|\varphi\|_{l^2_a}.$$\qed
\end{lemma}
Now let $\gamma=\gamma(t)$ be a $C^1$-function and let $U(t,\tau)v_0$ be a solution to
\begin{equation}
  \label{eq:p-LToda}
\left\{
  \begin{aligned}
& \frac{dv}{dt}=\dot{\gamma}JH''(u_{c_0}(\gamma))v,\\
& v(\tau)=\varphi.
  \end{aligned}
\right.
\end{equation}
If a modulation parameter $\gamma(t)$ is an increasing function and 
$\dot{\gamma}(t)$ is bounded away from $0$, we have the following.
\begin{corollary}
  \label{cor:1}
Let $c_0$, $a$, $b$ and $K$ be as in Lemma \ref{lem:linear-est}
and let $0\le T\le\infty$.
Suppose $\inf_{t\in[0,T]}\dot{\gamma}(t)\ge 1/2$, $\varphi\in l^2_a$
and
$ \la\varphi,J^{-1}\dot{u}_{c_0}(\gamma(\tau))\ra
=\la\varphi,J^{-1}\pd_cu_{c_0}(\gamma(\tau))\ra=0.$
Then
$$
\|U(t,\tau)\varphi\|_{X(t)}
\le Ke^{-b(t-\tau)/2}\|\varphi\|_{X(\tau)}
\quad\text{for $0\le \tau\le t\le T$,}$$
where $\|v\|_{X(t)}:=e^{-ac_0\gamma(t)}\|v\|_{l^2_a}$.
\end{corollary}
\begin{proof}
Let $s=\gamma(t)$, $\tau_1=\gamma(\tau)$ and $\tilde{v}(s)=v(\gamma^{-1}(s))$.
Then for $s\in[0,\gamma(T)]$,
$$ \frac{d\tilde{v}}{ds}=JH''(u_{c_0})\tilde{v}
\quad\text{and}\quad v(s)\in\operatorname{Range}Q_{c_0}(s).
$$
\par

Lemma \ref{lem:linear-est} and the fact that $\dot{\gamma}(t)\ge 1/2$
imply
\begin{align*}
\|v(t)\|_{X(t)}=
e^{-ac_0s}\|\tilde{v}(s)\|_{l^2_a}\le &
Ke^{-b(s-\tau_1)-ac_0\tau_1}\|\varphi\|_{l^2_a}
\\  \le &
Ke^{-b(t-\tau)/2}e^{-ac_0\gamma(\tau)}\|\varphi\|_{l^2_a}.  
\end{align*}
This completes the proof of Corollary \ref{cor:1}.
\end{proof}

We can estimate $\|v(t)\|_{l^2}$ by applying an argument from \cite{FP2} that
uses the convexity of Hamiltonian and the orthogonality condition
\eqref{eq:orth1}.
\begin{lemma}
  \label{lem:speed-Hamiltonian}
Let $u(t)$ be a solution to \eqref{eq:Toda2} satisfying
$u(0)=u_{c_0}(\tau_0)+v_0$. Then there exist positive numbers $\delta_2$ and
$C$ satisfying the following:
Suppose there exists $T\in[0,\infty]$ such that
$v(t)$ satisfies \eqref{eq:decomp} and \eqref{eq:orth1} for $t\in[0,T]$ 
and $\sup_{t\in[0,T]}|c(t)-c_0|+\|v_0\|_{l^2}\le \delta_2$. Then
\begin{equation}
  \label{eq:enorm}
\|v(t)\|_{l^2}^2\le C(|c(t)-c_0|+\|v_0\|_{l^2})\quad\text{for $t\in[0,T]$.}
\end{equation}
\end{lemma}
\begin{proof}
By \eqref{eq:orth1}, we have
$\la H'(u_{c(t)}(\gamma(t))),v(t)\ra
=\la J^{-1}\dot{u}_{c(t)}(\gamma(t)),v(t)\ra=0.$
Since $H(u(t))$ does not depend on $t$, 
it follows from the convexity of the functional $H$ and the above that
  \begin{align*}
    \delta H:=& H(u_{c_0}(\tau_0)+v_0)-H(u_{c_0})
\\ =& H(u_{c(t)}(\gamma(t))+v(t))-H(u_{c_0})
\\ =& H(u_{c(t)}) +\la H'(u_{c(t)}(\gamma(t))),v(t)\ra
+\frac12\la H''(u_{c(t)}(\gamma(t)))v(t),v(t)\ra
\\ & -H(u_{c_0})+O(\|v(t)\|_{l^2}^3)
\\ & \ge \frac12\|v(t)\|_{l^2}^2-C'|c(t)-c_0|+O(\|v(t)\|_{l^2}^3),
  \end{align*}
where $C'$ is a positive constant.
Noting that $|\delta H|=O(\|v_0\|_{l^2})$, we have \eqref{eq:enorm}
for a $C>0$.
\end{proof}

Because $l^2\subset l^r$ for every $r\in[2,\infty]$,
Lemma \ref{lem:speed-Hamiltonian} allows us to control every $l^r$-norm with
$r\ge2$.
\section{Proof of Theorem \ref{thm:1}}
\label{sec:3}
First, we derive from \eqref{eq:orth1} and \eqref{eq:orth2} a system of
ordinary differential equations which describe the motion of modulating speed
$c(t)$ and phase shift $x(t)=c(t)\gamma(t)$ of the main 
solitary wave.
\begin{lemma}
  \label{lem:modulation}
Let $u(t)$ be a solution to \eqref{eq:Toda2} and  $v_1(t)$ be a solution to
\eqref{eq:v1}.  
Suppose that $c$ and $\gamma$ are $C^1$-functions satisfying
\eqref{eq:orth1} and \eqref{eq:orth2} on $[0,T]$ and
$\inf_{t\in[0,T]}c(t)\linebreak >1$.
Then it holds for $t\in[0,T]$ that
\begin{align*}
& \dot{c}(t)=O(\|v_1(t)\|_{W(t)}^2+\|v_2(t)\|_{X(t)}^2),
\\ &  \dot{x}(t)-c(t)=
O(\|v_1(t)\|_{W(t)}+(\|v(t)\|_{l^2}+\|v_1(t)\|_{l^2})\|v_2(t)\|_{X(t)}),
  \end{align*}
where $\|u\|_{W(t)}
=\left(\sum_{n\in\Z}e^{-\kappa(c(t))|n-x(t)|}|u(n)|^2\right)^{1/2}$,
$\|u\|_{X(t)}=e^{-ax(t)}\|u\|_{l^2_a}$ and $a$ is a constant satisfying
$0<a\le\inf_{t\in[0,T]}\kappa(c(t))$.
\end{lemma}
\begin{proof}
  Differentiating \eqref{eq:orth1} with respect to $t$ and substituting \eqref{eq:v} into the resulting equation, we have
  \begin{align*}
&  \frac{d}{dt}\la v,J^{-1}\dot{u}_{c}(\gamma)\ra\\=& 
\la \dot{v},J^{-1}\dot{u}_{c}(\gamma)\ra
+\frac{\dot{x}}{c}\la v,J^{-1}\ddot{u}_{c}(\gamma)\ra
+\dot{c}\la v,J^{-1}\pd_c\dot{u}_{c}(\gamma)\ra
\\=&  \la JH''(u_{c}(\gamma))v,J^{-1}\dot{u}_{c}(\gamma))\ra+\la v,J^{-1}\ddot{u}_{c}(\gamma)\ra
\\ &+\la l_1+N_1,J^{-1}\dot{u}_{c}(\gamma)\ra
+\left(\frac{\dot{x}}{c}-1\right)\la v,J^{-1}\ddot{u}_{c}(\gamma)\ra
+\dot{c}\la v,J^{-1}\pd_c\dot{u}_{c}(\gamma)\ra
\\ =&0.
\end{align*}
Substituting $\ddot{u}_c=JH''(u_c)\dot{u}_c$ and $J^*=-J$ into the above,
we have
\begin{equation}
  \label{eq:modl1}
\dot{c}\left\{\frac{d}{dc}H(u_c)-\la v,J^{-1}\pd_c\dot{u}_{c}(\gamma)
\ra\right\}-\left(\frac{\dot{x}}{c}-1\right)\la v,J^{-1}\ddot{u}_c(\gamma)\ra  
= \la N_1,J^{-1}\dot{u}_c(\gamma)\ra.  
\end{equation}
\par

Differentiating \eqref{eq:orth2} with respect to $t$, we have
\begin{align*}
&  \frac{d}{dt}\la v_2,J^{-1}\pd_cu_c(\gamma)\ra
\\ =& 
\la \dot{v}_2,J^{-1}\pd_cu_c(\gamma)\ra
+\frac{\dot{x}}{c}\la v_2,J^{-1}\pd_c\dot{u}_c(\gamma)\ra
+\dot{c}\la v_2,J^{-1}\pd_c^2u_c(\gamma)\ra
\\=&  \la JH''(u_{c}(\gamma))v_2,J^{-1}\pd_cu_c(\gamma))\ra
+\la v_2,J^{-1}\pd_c\dot{u}_{c}(\gamma)\ra
\\ &+\la l_1+N_2,J^{-1}\pd_cu_c(\gamma)\ra
+\left(\frac{\dot{x}}{c}-1\right)\la v_2,J^{-1}\pd_c\dot{u}_c(\gamma)\ra
+\dot{c}\la v_2,J^{-1}\pd^2_cu_c(\gamma)\ra
\\ =&0.
\end{align*}
Substituting $\pd_c\dot{u}_c=JH''(u_c)\pd_cu_c$ into the above, we obtain
\begin{equation}
  \label{eq:modl2}
\begin{split}
&\left(\frac{\dot{x}}{c}-1\right)
\left\{\frac{d}{dc}H(u_c)+\la v_2,J^{-1}\pd_c\dot{u}_{c}(\gamma)\ra\right\}
+\dot{c}\la v_2,J^{-1}\pd_c^2u_{c}(\gamma)\ra
\\=& -\la N_2,J^{-1}\pd_cu_{c}(\gamma)\ra.
\end{split}  
\end{equation}
\par
Since $|N_1(t)|\lesssim |v(t)|^2$ and
$|J^{-1}\dot{u}_c(t,n)|\lesssim e^{-2\kappa(c)|n-x(t)|}$ as $n\to\infty$,
we have 
$$\la N_1,J^{-1}\dot{u}_c(\gamma)\ra=O(\|v(t)\|_{W(t)}^2).$$
Let $N_2(t)=\widetilde{N}_1(t)+\widetilde{N}_2(t)+\widetilde{N}_3(t)$, where
\begin{align*}
\widetilde{N}_1(t)=& N_1(t)-JH'(v(t))+Jv(t),
\\ \widetilde{N}_2(t)=& JH'(v(t))-JH'(v_1(t))-Jv_2(t),
\\ \widetilde{N}_3(t)=&
J\left(H''(u_{c(t)}(\gamma(t)))-1\right)v_1(t).  
\end{align*}
We put $G(v):=H'(v)-H'(0)-H''(0)v$ so that $JG(v)$ denotes a part of
$\widetilde{N}_1(t)$ that does not interact with the solitary wave
$u_{c}(\gamma)$. 
Since $|u_c(t,n)|\lesssim e^{-2\kappa(c)|n-x(t)|}$ and
$a\le \inf_{t\in[0,T]}\kappa(c(t))$, we have
$\|u_{c(t)}v^2\|_{X(t)}\lesssim \|v\|_{W(t)}^2.$
Hence by the definition of $\widetilde{N}_1$ and $\widetilde{N}_2$,
\begin{equation}
\label{eq:rhs1}
\|\widetilde{N}_1(t)\|_{X(t)}=\|N_1(t)-JG(v(t))\|_{X(t)}
\lesssim \|v(t)\|_{W(t)}^2,  
\end{equation}
\begin{align}
\label{eq:rhs2}
\|\widetilde{N}_2(t)\|_{X(t)}=&\|JG(v(t))-JG(v_1(t))\|_{X(t)} \\
\lesssim & (\|v(t)\|_{l^\infty}+\|v_1(t)\|_{l^\infty})\|v_2(t)\|_{X(t)}.
\notag
\end{align}
We see from \eqref{eq:q} and \eqref{eq:u} that 
$H''(u_c)-1$ decays like $e^{-2\kappa|n-x(t)|}$ as $n\to\pm\infty$ and for
$a\in(0,\kappa(c(t)))$,
\begin{equation}
  \label{eq:rhs2a}
\|\widetilde{N}_3(t)\|_{X(t)}\lesssim \|v_1(t)\|_{W(t)}.
\end{equation}
Let $\|u\|_{X(t)^*}=e^{ax(t)}\|u\|_{l^2_{-a}}$ and
$\|u\|_{W(t)^*}=
(\sum_{n\in\Z}e^{\kappa(c(t))|n-x(t)|}|u(n)|^2)^{1/2}.
$
In view of \eqref{eq:modl1}, \eqref{eq:modl2} and the fact that 
\begin{gather*}
\sup_{t\in[0,T]}\left(
\|J^{-1}\ddot{u}_{c(t)}(\gamma(t))\|_{W(t)^*}
+\|J^{-1}\pd_c\dot{u}_{c(t)}(\gamma(t))\|_{W(t)^*}\right)<\infty,\\
\sup_{t\in[0,T]}\left(\|J^{-1}\pd_c\dot{u}_{c(t)}(\gamma(t))\|_{X(t)^*}
+\|J^{-1}\pd_c^2u_{c(t)}(\gamma(t))\|_{X(t)^*}\right)<\infty,  
\end{gather*}
we have
\begin{equation*}
  \mathcal{A}(t)
  \begin{pmatrix}
    \dot{c}(t) \\ \dot{x}(t)-c(t)
  \end{pmatrix}=
  \begin{pmatrix}
O(\|v(t)\|_{W(t)}^2) \\ 
O(\|v_1(t)\|_{W(t)}+
(\|v(t)\|_{l^2}+\|v_1(t)\|_{l^2})\|v_2(t)\|_{X(t)})
  \end{pmatrix},
\end{equation*}
where $\mathcal{A}(t)=\operatorname{diag}(dH(u_c)/dc,dH(u_c)/dc)
+O(\|v_1(t)\|_{W(t)}+\|v_2(t)\|_{X(t)})$.
We have thus proved Lemma \ref{lem:modulation}.
\end{proof}

Since $v_1(t)$ is smaller than the main wave, it moves more slowly and will be separated from the main wave. The following is an analog of
\textit{virial lemma} for small solutions in Martel and Merle \cite{MM2}.
\begin{lemma}
  \label{lem:monotonicity}
Let $v_1(t)$ be a solution to \eqref{eq:v1}. 
\begin{itemize}
\item [(i)] Suppose $v_0\in l^2$. Then
$\sup_{t\in\R}\|v_1(t)\|\le C\|v_0\|_{l^2},$ where $C$ can be chosen as an
increasing function of $\|v_0\|_{l^2}$.
\item[(ii)]
Let $c_1>1$ and $\tilde{x}(t)$ be a $C^1$-function satisfying
$\inf_{t\in\R}\tilde{x}_t\ge c_1$.
Then there exist positive numbers $a_0$ and $\delta_3$ such that 
if $a\in(0,a_0)$ and $\|v_0\|_{l^2}\le\delta_3$,
$$\|\psi_a(t)^{1/2}v_1(t)\|_{l^2}^2+
\int_0^t\|\tilde{\psi}_a(t)v_1(s)\|_{l^2}^2ds \lesssim
\|\psi_a(0)^{1/2}v_0\|_{l^2}^2,$$
where $\psi_a(t,x)=1+\tanh a(x-\tilde{x}(t))$ and
$\tilde{\psi}_a(t,x)=a^{1/2}\sech a(x-\tilde{x}(t))$. 
\end{itemize}
\end{lemma}
\begin{corollary}
  \label{cor:monotonicity}
Let $v_1(t)$ be a solution to \eqref{eq:v1}. 
For every $c_1>1$, there exists $\delta_3>0$ such that
$\lim_{t\to\infty}\|v_1(t)\|_{l^2(n\ge c_1t)}=0$
if $\|v_0\|_{l^2}<\delta_3$.
\end{corollary}
\begin{proof}[Proof of Lemma \ref{lem:monotonicity}]
Since $v_1(t)\in C^2(\R;l^2)$ is a solution to \eqref{eq:v1},
we have $H(v_1(t))=H(v_0)$ for $t\in\R$.
Noting that $V(x)$ is coercive and $\inf_{|x|\le R}|x|^{-2}V(x)\linebreak >0$ for
 every $R>0$, we have 
$$\delta' \|v(t)\|_{l^2}^2\le H(v(t))
=H(v_0)\le C(\|v_0\|_{l^2})\|v_0\|_{l^2}^2,$$
where $C$ can be chosen as an increasing function of $\|v_0\|_{l^2}$
and $\delta'$ is a positive constant depending only on $\|v_0\|_{l^2}$.
\par

Next, we prove (ii).
Put 
$$v_1(t)={}^t\!(r_1(t,n),p_1(t,n)),\quad
h_1(t,n)=\frac12p_1(t,n)^2+V(r_1(t,n)).$$
By \eqref{eq:Toda2} and the fact that there exists a $C>0$ such that
for every $n\in\Z$,
\begin{align*}
&\left|V(r_1(t,n))-\frac{r_1(t,n)^2}{2}\right|
\le   C\|v_0\|_{l^2}|r_1(t,n)|^2,
\\ & \left|V'(r_1(t,n))-r_1(t,n)\right| \le C\|v_0\|_{l^2}|r_1(t,n)|,
\end{align*}
we have
\begin{align*}
&  \frac{d}{dt}\sum_{n\in\Z}\psi_a(t,n)h_1(t,n)
\\ =&
\sum_{n\in\Z}p_1(t,n)V'(r_1(t,n-1))\left(\psi_a(t,n-1)-\psi_a(t,n)\right)\
 +\sum_{n\in\Z}\pd_t\psi_a(t,n)h_1(t,n)
\\ \le &
-\frac{\tilde{x}_t(t)}2\sum_{n\in\Z}\tilde{\psi}_a(t,n)^2p_1(t,n)^2
\\ & +(1+C'\|v_0\|_{l^2})
\sum_{n\in\Z}\left|\psi_a(t,n-1)-\psi_a(t,n)\right||p_1(t,n)r_1(t,n-1)|
\\ &-\frac{\tilde{x}_t(t)}2(1-C'\|v_0\|_{l^2})
\sum_{n\in\Z}\tilde{\psi}_a(t,n-1)^2r_1(t,n-1)^2,
\end{align*}
where $C'$ is a positive constant.
Let $\delta_3$ and $a$ be sufficiently small numbers.
Since $\inf \tilde{x}_t\ge c_1>1$ and
$$\sup_{n,t}\left|\frac{\psi_a(t,n)-\psi_a(t,n-1)}{\tilde{\psi}_a(t,n)^2}-1\right|
=O(a) \quad\text{as $a\downarrow 0$},$$
there exists a $\tilde{\delta}>0$ such that for $t\in[0,T]$,
\begin{equation}
\label{eq:mono1}
 \frac{d}{dt}\sum_{n\in\Z}\psi_a(t,n)h_1(t,n)
 \le  -\tilde{\delta}
\sum_{n\in\Z}\tilde{\psi}_a(t,n)^2(p_1(t,n)^2+r_1(t,n)^2).
\end{equation}
Integrating \eqref{eq:mono1} over $[0,t]$, we have
\begin{align*}
& \sum_{n\in\Z}\psi_a(t,n)h_1(t,n)+ 
\tilde{\delta}\sum_{n\in\Z}\int_0^t \tilde{\psi}_a(s,n)^2(p_1(s,n)^2+r_1(s,n)^2)ds
\\ \lesssim & \sum_{n\in\Z}\psi_a(0,n)h_1(0,n)
\lesssim  \|v_0\|_{l^2}^2.
\end{align*}
We have thus proved Lemma \ref{lem:monotonicity}.
\end{proof}
\begin{proof}[Proof of Corollary \ref{cor:monotonicity}]
  Let $c_2\in(1,c_1)$ and let $\tilde{x}(t)=c_2t$.
Then by Lemma \ref{lem:monotonicity}, we have
$
\|v_1(t)\|_{l^2(n\ge c_2t)}\lesssim \|\psi_a(0)^{1/2}v_0\|_{l^2}.$
Let $n_0(t)=[(c_1-c_2)t]$, a largest integer which is smaller than
$(c_1-c_2)t$. Then we have $n_0(t)\to\infty $ as $t\to\infty$ and
\begin{align*}
\|v_1(t)\|_{l^2(n\ge c_1t)}\le &
 \|v_1(t,\cdot+n_0(t))\|_{l^2(n\ge c_2t)}
\\ \lesssim &
\|\psi_a(0,\cdot)^{1/2}v_0(\cdot+n_0(t))\|_{l^2}.
\end{align*}
Letting $t\to\infty$, we have
$\lim_{t\to\infty}\|v_1(t)\|_{l^2(n\ge c_1t)}=0.$
This completes the proof of Corollary \ref{cor:monotonicity}.
\end{proof}

Next, we will show the decay estimate of $v_2$.
\begin{lemma}
\label{lem:v2}
Let $c_0>1$, $a\in(0,\kappa(c_0)/3)$ and $\delta_4$ be a sufficiently small
positive number.
Suppose that the decomposition \eqref{eq:decomp}, \eqref{eq:orth1} and 
\eqref{eq:orth2} exists for $t\in[0,T]$ and that
$\|v_0\|_{l^2}+\sup_{t\in[0,T]}\left(|c(t)-c_0|+|\dot{x}(t)-c_0|\right)
\le\delta_4$,
where $x(t)=c(t)\gamma(t)$. Then
\begin{equation}
  \label{eq:v-w1}
\|Q_{c(t)}(\gamma(t))v_2(t)\|_{X(t)} \le  C
\left( e^{-bt/4}\|v_0\|_{l^2}+\int_0^te^{-b(t-s)/4}\|v_1(s)\|_{W(s)}ds\right),
\end{equation}  for $t\in[0,T]$, and
\begin{equation}
  \label{eq:v-w2}
\int_0^T\|v_2(t)\|_{X(t)}^2dt\le C\|v_0\|_{l^2}^2,
\end{equation}
where $C$ is a positive constant independent of $T$ and
$\|v\|_{X(t)}$ and $\|v\|_{W(t)}$ are as in Lemma \ref{lem:modulation}.
\end{lemma}
\begin{proof}
Let $\tilde{v}_2(t):=Q_{c(t)}(\gamma(t))v_2(t)$ and
$w(t)=Q_{c_0}(\tilde{\gamma}(t))\tilde{v}_2(t)$,
where $\tilde{\gamma}(t)=x(t)/c_0$. Here we choose $\tilde{\gamma}(t)$ so that
$u_{c(t)}(\gamma(t))$ and $u_{c_0}(\tilde{\gamma}(t))$ have the same phase
shift and for $0<a<\min(\kappa(c(t)),\kappa(c_0))$,
$$\|Q_{c(t)}(\gamma(t))-Q_{c_0}(\tilde{\gamma}(t))\|_{B(l^2_a)}=O(|c(t)-c_0|).
$$
By \eqref{eq:orth1} and \eqref{eq:orth2}, 
\begin{equation}
  \label{eq:deftv2}
\begin{split}
\tilde{v}_2(t)=&
v_2(t)-\theta(c(t))\la v_2(t),J^{-1}\dot{u}_c(\gamma(t))\ra\pd_cu_c(\gamma(t))
\\=& v_2(t)+\theta(c(t))\la v_1(t),J^{-1}\dot{u}_c(\gamma(t))\ra
\pd_cu_c(\gamma(t)).
\end{split}   
\end{equation}
Thus we have
$$
\frac{d\tilde{v}_2}{dt}=JH''(u_{c(t)})(\gamma(t))\tilde{v}_2
+l_1(t)+l_2(t)+l_3(t)+N_2(t),
$$
where
\begin{align*}
l_2(t)=& \frac{d}{dt}\left\{\theta(c(t))
\la v_1(t),J^{-1}\dot{u}_{c(t)}(\gamma(t))\ra\pd_cu_{c(t)}(\gamma(t))\right\},
\\
l_3(t)=& -\theta(c(t))\la v_1(t),J^{-1}\dot{u}_{c(t)}(\gamma(t))\ra
\pd_c\dot{u}_{c(t)}(\gamma(t)).  
\end{align*}
Since
$$\left[\frac{d}{dt}-\dot{\tilde{\gamma}}JH''(u_{c_0}(\tilde{\gamma}))
,Q_{c_0}(\tilde{\gamma})\right]=0,$$
we have
\begin{equation}
  \label{eq:w}
  \begin{split}
 \dot{w}-\dot{\tilde{\gamma}}JH''(u_{c_0}(\tilde{\gamma}))w
=&
Q_{c_0}(\tilde{\gamma})
\left\{\dot{\tilde{v}}_2-
\dot{\tilde{\gamma}}JH''(u_{c_0}(\tilde{\gamma}))\tilde{v}_2\right\}
\\=& 
Q_{c_0}(\tilde{\gamma})
\left\{\sum_{1\le k\le 4}l_k+\sum_{1\le k\le 3}\widetilde{N}_k\right\},
  \end{split}
\end{equation}
where
\begin{align*}
l_4(t)=& J\left\{H''(u_{c(t)}(\gamma(t)))
-H''(u_{c_0}(\tilde{\gamma}(t)))\right\}\tilde{v}_2(t)
\\ & -(\dot{\tilde{\gamma}}(t)-1)
JH''(u_{c_0}(\tilde{\gamma}(t)))\tilde{v}_2(t).
\end{align*}
In view of  Lemma \ref{lem:modulation}, we have for $a\in(0,2\kappa(c(t)))$,
\begin{equation}
  \label{eq:rhs2b}
\|l_1\|_{X(t)} \lesssim  \|v_1(t)\|_{W(t)}+
(\|v(t)\|_{l^2}+\|v_1(t)\|_{l^2}+\|v_2(t)\|_{X(t)})\|v_2(t)\|_{X(t)}.
\end{equation}
By \eqref{eq:v1} and the fact that $J^{-1}\dot{u}_c(\gamma)$,
$\frac{d}{dt}J^{-1}\dot{u}_c(\gamma)$, $\pd_cu_c(\gamma)$ and
$\frac{d}{dt}\pd_cu_c(\gamma)$ decay like
$e^{-2\kappa|n-x(t)|}$ as $n\to\pm\infty$,
we have
\begin{equation}
  \label{eq:rhs2c}
\|l_2(t)\|_{X(t)}\lesssim  \|v_1(t)\|_{W(t)}.
\end{equation}
Similarly, we have
\begin{equation}
  \label{eq:rhs3}
 \|l_3(t)\|_{X(t)}\lesssim \|v_1(t)\|_{W(t)}.
\end{equation}
Since $x(t)=c_0\tilde{\gamma}(t)=c(t)\gamma(t)$,
\begin{equation}
  \label{eq:rhs4}
  \begin{split}
\|l_4(t)\|_{X(t)}\lesssim &(|c(t)-c_0|+|\dot{x}(t)-c_0|)
\|\tilde{v}_2(t)\|_{X(t)}
\\ \lesssim & \delta_4 (\|v_1(t)\|_{W(t)}+\|v_2(t)\|_{X(t)}).      
  \end{split}
\end{equation}
Let $U(t,s)$ be a flow generated by 
$$
\frac{dw}{dt}=\dot{\tilde{\gamma}}(t)JH''(u_{c_0}(\tilde{\gamma}(t)))w.$$
Applying Corollary \ref{cor:1} to \eqref{eq:w} and substituting
\eqref{eq:rhs1}--\eqref{eq:rhs2a} and \eqref{eq:rhs2b}--\eqref{eq:rhs4},
 we have
\begin{align*}
& \|w(t)\|_{X(t)}
\\ \lesssim & \|U(t,0)w(0)\|_{X(t)}
+\sum_{k=1}^4\int_0^t\|U(t,s)Q_{c_0}(\tilde{\gamma}(s))l_k(s)\|_{X(t)}
\\ & +
\sum_{k=1}^3\int_0^t\|U(t,s)Q_{c_0}(\tilde{\gamma}(s))
\widetilde{N}_k(s)\|_{X(t)}
\\ \lesssim & e^{-bt/2}\|w(0)\|_{X(0)}
+\int_0^t e^{-b(t-s)/2} \|v_2(s)\|_{X(s)}^2ds
\\ +& \int_0^t e^{-b(t-s)/2}\left\{ \|v_1(s)\|_{W(s)}+
(\delta_4+\|v_1(s)\|_{l^2}+\|v_2(s)\|_{l^2})\|v_2(s)\|_{X(s)}\right\}.
\end{align*}
Here we use $\|u\|_{W(t)}\lesssim \|u\|_{l^2}$ and 
$\|u\|_{W(t)}\lesssim \|u\|_{X(t)}$ for $a\in(0,\kappa(c(t))/2)$.
By the definition of $\tilde{v}_2$ and $w$,
\begin{equation}
\label{eq:v2-1}
  \|v_2(t)\|_{X(t)}\lesssim \|\tilde{v}_2(t)\|_{X(t)}+
\|v_1(t)\|_{W(t)},
\end{equation}
\begin{align}
\label{eq:v2-2}
  \|\tilde{v}_2(t)\|_{X(t)}\le & \|w(t)\|_{X(t)}
+\|(Q_{c(t)}(\gamma(t))-Q_{c_0}(\tilde{\gamma}(t)))\tilde{v}_2(t)\|_{X(t)}
\\ \lesssim & \notag
\|w(t)\|_{X(t)}+|c(t)-c_0|\|\tilde{v}_2(t)\|_{X(t)}.
\end{align}
If $\delta_4$ is sufficiently small, Eqs. \eqref{eq:v2-1} and \eqref{eq:v2-2}
imply
$\|\tilde{v}_2(t)\|_{X(t)}\lesssim \|w(t)\|_{X(t)}$ and
\begin{equation}
  \label{eq:v2-3}
\|v_2(t)\|_{X(t)}\lesssim \|w(t)\|_{X(t)}+\|v_1(t)\|_{W(t)}.  
\end{equation}
It follows from Lemmas \ref{lem:monotonicity} (i) and
 \ref{lem:speed-Hamiltonian} that
$\|v_1(t)\|_{l^2}+\|v_2(t)\|_{l^2}^2\lesssim \|v_0\|_{l^2}
+|c(t)-c_0|.$
Thus as long as $\sup_{0\le s\le t}\|w(s)\|_{X(s)}\le \sqrt{\delta_4}$,
we have
\begin{align*}
\|w(t)\|_{X(t)}\lesssim & e^{-bt/2}\|w(0)\|_{X(0)}
\\ &+\int_0^t e^{-b(t-s)/2}\left(\|v_1(s)\|_{W(s)}+
\sqrt{\delta_4}\|w(s)\|_{X(s)}\right)ds.  
\end{align*}
Applying Gronwall's inequality, we have
\begin{align}
\label{eq:gr}
\|w(t)\|_{X(t)}\lesssim & e^{-(b/2+O(\sqrt{\delta_4}))t}\|w(0)\|_{X(0)}
\\ & +\int_0^t e^{-(b/2+O(\sqrt{\delta_4}))(t-s)}\|v_1(s)\|_{W(s)}ds.
\notag
\end{align}
By the definition of $w$, \eqref{eq:v2}, \eqref{eq:deftv2} and
Lemma \ref{lem:decomp},
\begin{equation}
  \label{eq:gr2}
\|w(0)\|_{X(0)} \lesssim
\|v_2(0)\|_{X(0)}+\|v_1(0)\|_{l^2}\lesssim  \|v_0\|_{l^2}.  
\end{equation}
In view of Lemma \ref{lem:monotonicity} (i),  \eqref{eq:gr} and \eqref{eq:gr2},
we have $\|w(t)\|_{X(t)}\lesssim \|v_0\|_{l^2}=O(\delta_4)$ and \eqref{eq:gr}
persists for $t\in[0,T]$ if $\delta_4$ is sufficiently small. Thus by 
\eqref{eq:v2-2}, we have \eqref{eq:v-w1}
Combining \eqref{eq:v-w1}, \eqref{eq:deftv2}, and
Lemma \ref{lem:monotonicity} (ii) and using Young's inequality, we have
\begin{align*}
\|v_2(t)\|_{L^2(0,T;X(t))} \lesssim &  \|v_0\|_{l^2}
+\|e^{-bt/4}\|_{L^1(0,T)} \|v_1\|_{L^2(0,T;W(t))}
\\ \lesssim & \|v_0\|_{l^2}.
\end{align*}
We have thus completed the proof of Lemma \ref{lem:v2}.
\end{proof}

Now, we are in position to prove the following proposition.
\begin{proposition}
  \label{prop:2}
Let $c_0>1$, $\tau_0\in\R$ and let $u(t)$ be a solution to \eqref{eq:Toda2}
with $u(0)=u_{c_0}(\tau_0)+v_0$.
For every $\varepsilon>0$, there exists a positive number $\delta>0$ satisfying
the following: If $\|v_0\|_{l^2}<\delta$, there exist a constant $c_+>1$
and a $C^1$-function $x(t)$ such that
\begin{align}
& \|u(t)-\tilde{u}_{c_0}(\cdot-x(t))\|_{l^2}< \varepsilon,\\
& 
\label{eq:asym-st}
\lim_{t\to\infty}\left\|
u(t)-\tilde{u}_{c_+}(\cdot-x(t))\right\|_{l^2(n\ge x(t)-R)}=0
\quad\text{for every $R>0$},
\\ \label{eq:cbound}
& \sup_{t\in\R}\left(|c(t)-c_0|+|\dot{x}(t)-c_0|\right)=O(\|v_0\|_{l^2}),
\\ \label{eq:clim}
& \lim_{t\to\infty}c(t)=c_+,\quad \lim_{t\to\infty}\dot{x}(t)=c_+.
\end{align}
\end{proposition}
\begin{proof}
 Let $\delta_5=\min_{1\le i\le 4}\delta_i$ and 
 \begin{align*}
&  T_0:=\sup\left\{t:
\text{\eqref{eq:decomp}, \eqref{eq:orth1} and \eqref{eq:orth2} hold
for $0\le \tau\le t$}\right\}, 
\\ &
T_1:=\sup\left\{t\le T_0: \|v_0\|_{l^2}+ 
\sup_{0\le \tau\le t}\left(|c(\tau)-c_0|+|\dot{x}(\tau)-c_0|\right)
\le \delta_5\right\}.
\end{align*}
If $\delta$ is sufficiently small, Proposition \ref{prop:1} and Lemma
\ref{lem:decomp} imply that $T_1>0$.
We will show that $T_0=T_1$ for small $\delta$.
Suppose that $t\in[0,T_1)$. 
Lemmas  \ref{lem:modulation}, \ref{lem:monotonicity} and \ref{lem:v2} and 
\eqref{eq:v2-1} imply
\begin{align}
  \label{eq:xdot}
|\dot{x}(t)-c(t)|\lesssim & \|v_1(t)\|_{W(t)}+\|v_2(t)\|_{X(t)}
\\ \lesssim & \|v_1(t)\|_{W(t)}+\|\tilde{v}_2(t)\|_{X(t)} 
\lesssim \|v_0\|_{l^2}. \notag
\end{align}
\par
By  Lemmas \ref{lem:decomp} and \ref{lem:modulation},
\begin{align*}
    |c(t)-c_0|\le & |c(0)-c_0|+\int_0^t|\dot{c}(s)|ds
\\ \lesssim & \|v_0\|_{l^2}+\int_0^t
\left(\|v_1(s)\|_{W(s)}^2+\|v_2(s)\|_{X(s)}^2\right)ds.  
\end{align*}
In view of Lemmas \ref{lem:monotonicity} (ii) and
\ref{lem:v2},
we have
\begin{equation}
  \label{eq:c-est}
|c(t)-c_0|\lesssim \|v_0\|_{l^2}.
\end{equation}
It follows from \eqref{eq:xdot} and \eqref{eq:c-est} that $T_0=T_1$ if
$\delta$ is sufficiently small.
\par
Next, we will show that $T_0=\infty$ for small $\delta$. 
Suppose that for every $\delta>0$, there exists $v_0$ 
such that $\|v_0\|_{l^2}<\delta$ and $T_0<\infty$.
By Lemma \ref{lem:speed-Hamiltonian} and \eqref{eq:c-est},
\begin{equation}
  \label{eq:vl}
\sup_{t\in[0,T_0]}\|v(t)\|_{l^2}^2 \lesssim \|v_0\|_{l^2}.
\end{equation}
Using \eqref{eq:v2-1}, Lemmas \ref{lem:v2} and \ref{lem:monotonicity} (i),
we have 
\begin{equation}
  \label{eq:vw}
\begin{split}
\sup_{t\in[0,T_0]}\|v_2(t)\|_{X(t)}\lesssim & 
\sup_{t\in[0,T_0]}\left(\|v_1(t)\|_{W(t)}+\|\tilde{v}_2(t)\|_{X(t)}\right)
 \lesssim  \|v_0\|_{l^2}.
\end{split}  
\end{equation}
By \eqref{eq:vl} and \eqref{eq:vw}, we get
$\|v(T_0)\|_{l^2}+e^{-ax(T_0)}\|v_2(T_0)\|_{l^2_a} \lesssim \|v_0\|_{l^2}.$
Hence it follows from Lemma \ref{lem:decomp} that the decomposition
\eqref{eq:decomp}, \eqref{eq:orth1} and \eqref{eq:orth2} can be extended
beyond $t=T_0$ if $\|v_0\|_{l^2}$ is small.
This is a contradiction. Thus we prove $T_0=\infty$ for small  $v_0\in l^2$.
\par
Let $\delta$ be a small positive number such that $T_0=T_1=\infty$.
Then Lemma \ref{lem:monotonicity} (ii) and Lemma \ref{lem:v2} imply
$\|v_1(t)\|_{W(t)}+\|v_2(t)\|_{X(t)}\in L^2(0,\infty)$.
Thus by Lemma \ref{lem:modulation}, we see that
$\dot{c}(t)$ is integrable on $[0,\infty)$ and that there exists $c_+$
satisfying $\lim_{t\to\infty}c(t)=c_+$.
\par
Next, we will prove \eqref{eq:asym-st}.
As in the proof of Corollary \ref{cor:monotonicity}, we can prove
$\lim_{t\to\infty}\|v_1(t)\|_{W(t)}=0$. Combining this with \eqref{eq:v2decay},
we have
\begin{equation}
  \label{eq:==}
  \lim_{t\to\infty}\dot{x}(t)=\lim_{t\to\infty}c(t)=c_+.
\end{equation}
By \eqref{eq:v2-1}, Lemma \ref{lem:v2} and the fact that
$\|v_1(t)\|_{W(t)}\in L^2(0,\infty)$,
\begin{equation}
  \label{eq:v2decay}
\begin{split}
  \|v_2(t)\|_{X(t)}
\lesssim & \|v_1(t)\|_{W(t)}+ \|\tilde{v}_2(t)\|_{X(t)}
\\ \lesssim &
\|v_1(t)\|_{W(t)}+e^{-bt/4}\|v_0\|_{l^2} +\sup_{t/2\le s\le t}\|v_1(s)\|_{W(s)}
\\ &
+e^{-bt/8}\left(\int_0^{t/2}\|v_1(s)\|_{W(s)}^2ds\right)^{1/2}
 \to 0 \quad\text{as $t\to\infty$.}
\end{split}  
\end{equation}
Since $\|v_2(t)\|_{l^2(n\ge x(t)-R)} \lesssim \|v_2(t)\|_{X(t)}$ for every 
$R>0$, Corollary \ref{cor:monotonicity} and \eqref{eq:v2decay} imply
\eqref{eq:asym-st}.
Combining this \eqref{eq:==} and \eqref{eq:v2decay},
we have
$$\lim_{t\to\infty}\dot{x}(t)=\lim_{t\to\infty}c(t)=c_+.$$
\par
We have thus completed the proof of Proposition \ref{prop:2}.
\end{proof}
Combining Proposition \ref{prop:2} and the \textit{monotonicity} argument
given in \cite{MM2}, we obtain Theorem \ref{thm:1}.
\begin{proof}[Proof of Theorem \ref{thm:1}]
Put
\begin{gather*}
{}^{t}\!(\tilde{r}(t,n),\tilde{p}(t,n)):=v(t,n),
\quad
h(t,n)=\frac{1}{2}\tilde{p}(t,n)^2+V(\tilde{r}(t,n)),\\
N_3(t)=J\left\{H'(u_{c(t)}(\gamma(t))+v(t))-H'(u_{c(t)}(\gamma(t)))
-H'(v(t))\right\}.
\end{gather*}
Let $\sigma\in (1,c_+)$, $t_1\ge0$ and
$\tilde{x}(t)=x(t_1)+\sigma(t-t_1)$. Let $\psi_a(t,n)$ and
$\tilde{\psi}_a(t,n)$ be as in Lemma \ref{lem:monotonicity}.
Then
\begin{align*}
& \frac{d}{dt}\sum_{n\in\Z}\psi_a(t,n)h(t,n)
\\=& \la H'(v(t)),\psi_a(t)\dot{v}(t)\ra+\sum_{n\in\Z}\pd_t\psi_a(t,n)h(t,n)
\\ =&
\sum_{n\in\Z}\tilde{p}(t,n)V'(\tilde{r}(t,n-1))(\psi_a(t,n-1)-\psi(t,n))
\\ & +\la \psi_a(t)l_1(t),H'(v(t))\ra +\la \psi_a(t)N_3(t),H'(v(t))\ra 
+\sum_{n\in\Z}\pd_t\psi_a(t,n)h(t,n).
\end{align*}
Here we use $\frac{dv}{dt}=JH'(v(t))+l_1(t)+N_3(t)$.
Suppose that $a>0$ and $\|v_0\|_{l^2}$ are sufficiently small.
Since $\|v(t)\|_{l^2}\lesssim\|v_0\|_{l^2}$ follows from Proposition
\ref{prop:2}, we see that there exists a $\delta'>0$
\begin{align*}
\frac{d}{dt}\sum_{n\in\Z}\psi_a(t,n)h(t,n)
 \le &
-\delta'\sum_{n\in\Z}\tilde{\psi}_a(t,n)
\left(\tilde{r}(t,n)^2+\tilde{p}(t,n)^2\right)\\
& +\la \psi_a(t)l_1(t),H'(v(t))\ra+\la \psi_a(t)N_3(t),H'(v(t))\ra
\end{align*}
in exactly the same way as the proof of Lemma \ref{lem:monotonicity}.
By the definitions of $l_1(t)$ and $N_3(t)$ and Lemma \ref{lem:modulation},
\begin{align*}
& |N_3(t)|\lesssim |u_{c(t)}(\gamma(t))v(t)|,\\
& |\la l_1(t),H'(v(t))\ra|\lesssim (\|v_1(t)\|_{W(t)}+\|v_2(t)\|_{X(t)})^2.
\end{align*}
Combining the above, we have
\begin{align*}
& \sum_{n\in\Z}\psi_a(t,n)\left(\tilde{r}(t,n)^2+\tilde{p}(t,n)^2\right) \\
\lesssim & \sum_{n\in\Z}\psi_a(t_1,n)h(t_1,n)
+\int_{t_1}^t(\|v_1(s)\|_{W(s)}+\|v_2(s)\|_{X(s)})^2ds
\\ \lesssim & \sum_{n\in\Z}\psi_a(t_1,n)
\left(|v_1(t_1,n)|^2+|v_2(t_1,n)|^2\right)
+\int_{t_1}^t(\|v_1(s)\|_{W(s)}+\|v_2(s)\|_{X(s)})^2ds.
\end{align*}
As in the proof of Corollary \ref{cor:monotonicity},
we have
$$\lim_{t_1\to\infty}\sum_{n\in\Z}\psi_a(t_1,n)|v_1(t_1,n)|^2=0.$$
On the other hand, Lemma \ref{lem:v2} implies
$$\sum_{n\in\Z}\psi_a(t_1,n)|v_2(t_1,n)|^2\lesssim
\|v_2(t_1)\|_{X(t_1)}^2\to0\quad \text{as $t_1\to\infty$.}$$
Furthermore, Lemmas \ref{lem:monotonicity} and \ref{lem:v2} and 
Proposition \ref{prop:2} imply
$$
\lim_{t_1\to\infty}\int_{t_1}^\infty
(\|v_1(s)\|_{W(s)}^2+\|v_2(s)\|_{X(s)}^2)ds=0.$$
Combining the above, we obtain
$$\lim_{t_1\to\infty}\sup_{t\ge t_1}
\|v(t)\|_{l^2(n\ge \sigma t)}=0.$$
Thus we complete the proof of Theorem \ref{thm:1}.
\end{proof}
\section{Proof of Theorem \ref{thm:2}}
\label{sec:4}
In this section, we will prove orbital and asymptotic stability of solitary
waves to FPU lattice \eqref{eq:FPU}.
For a two-parameter family of solitary wave solutions
$\{u_c(t+\delta): c\in[c_1,c_2],\;\delta\in\R\}$ that satisfies the condition
(P1)--(P4) below, we can prove the orbital and asymptotic stability of
solitary wave solutions in exactly the same way as Theorem \ref{thm:1}.
\begin{itemize}
\item[(P1)] There exists an open interval $I$ such that
$V''(r)>0$ for every $r\in I$ and that 
$\overline{\{r_c(x):x\in\R\}}\subset I$ for every $c\in[c_1,c_2]$.
\item[(P2)] 
There exists $a>0$ such that the map
$\R\times [c_1,c_2]\ni(t,c)\mapsto u_c(t)\in l^2_a\cap l^2_{-a}$ is $C^2$.
\item[(P3)] 
The solitary wave energy $H_F(u_c)$ satisfies
$dH_F(u_c)/dc\ne 0$ for $c\in[c_1,c_2]$.
\item[(P4)]
Let $c_0\in[c_1,c_2]$ and $a\in(0,2\kappa(c_0/c_s))$.
Let $U_0(t,\tau)\varphi$ be a solution to
\begin{equation}
  \label{eq:LFPU}
\left\{
  \begin{aligned}
& \frac{dv}{dt}=JH_F''(u_{c_0})v.\\
& v(\tau)=\varphi.
  \end{aligned}
\right.
\end{equation}
Then there exist positive numbers $b$ and $K$ such that for every
$\varphi\in l^2_a$ and $t\ge \tau$,
$$e^{-ac_0(t-\tau)}\|U_0(t,\tau)Q_c(\tau)\varphi\|_{l^2_a}
\le Ke^{-b(t-\tau)}\|\varphi\|_{l^2_a}.$$
\end{itemize}
\begin{proof}[Proof of Theorem \ref{thm:2}]
If $c>c_s$ and $c$ is sufficiently close to $c_s$, there exists a unique
solitary wave solution to \eqref{eq:solitary} up to translation
(\cite[Theorem 1.1]{FP1}). By \cite[Theorem 1.1]{FP1},
we see that a solitary wave solution satisfies (P1) and (P3) if $c$
is close to $c_s$.
Slightly modifying the proof of \cite[Proposition 6.1]{FP1} and 
\cite[Proposition A.3]{FP2}, we obtain (P2).
Since (P4) holds for small solitary waves (see \cite{FP4}),
Theorem \ref{thm:2} can be proved in exactly the same way as Theorem \ref{thm:1}.
\end{proof}
\section*{Acknowledgment}
The author would like to express his gratitude to Professor Robert L. Pego 
for his hospitality at Carnegie Mellon University where this work was
carried out.

\end{document}